\documentclass[10pt,reqno]{amsart}
\usepackage{latexsym,amssymb,amsmath}
\usepackage{amssymb,amsthm,upref,amscd}
\usepackage{soul,cite,mathrsfs}
\usepackage{color}
\usepackage[T1]{fontenc}
\usepackage{enumitem,graphicx}
\usepackage[colorlinks=true,urlcolor=blue,
citecolor=red,linkcolor=blue,linktocpage,pdfpagelabels,
bookmarksnumbered,bookmarksopen]{hyperref}
\usepackage[english]{babel}

\usepackage[hyperpageref]{backref}

\numberwithin{equation}{section}

\newtheorem{thm}{Theorem}[section]
\newtheorem{lem}[thm]{Lemma}

\newtheorem{remark}[thm]{Remark}

\newcommand{\dsp}{\displaystyle}
\newcommand{\C}{\mathbb{C}}
\newcommand{\N}{\mathbb{N}}

\title[Fractional magnetic Schr\"{o}dinger equations ]{Fractional magnetic Schr\"{o}dinger equations with potential vanishing at infinity and supercritical exponents}

\author[J.C. \ de Albuquerque]{Jos\'e Carlos de Albuquerque}
\author[J.L. Santos]{Jos\'{e} Luando Santos}

\address[J.C. de~Albuquerque]{Departamento de Matem\'atica, Universidade Federal de Pernambuco
	\newline\indent
	50670-901, Recife-PE, Brazil}
\email{josecarlos.melojunior@ufpe.br}

\address[J.L. Santos]{Departamento de Matem\'atica, Universidade Federal de Pernambuco
	\newline\indent
	50670-901, Recife-PE, Brazil}
\email{luando.brito@ufpe.br}

\thanks{Corresponding author: J.C. de Albuquerque}

\begin{document}
\begin{abstract}
This paper focuses on the following class of fractional magnetic Schr\"{o}dinger equations
\begin{equation*}
	(-\Delta)_{A}^{s}u+V(x)u=g(\vert u\vert^{2})u+\lambda\vert u\vert^{q-2}u, \quad \mbox{in } \mathbb{R}^{N},
\end{equation*}
where $(-\Delta)_{A}^{s}$ is the fractional magnetic Laplacian, $A :\mathbb{R}^N \rightarrow \mathbb{R}^N$ is the magnetic potential, $s\in (0,1)$, $N>2s$, $\lambda \geq0$ is a parameter, $V:\mathbb{R}^N \rightarrow \mathbb{R}$ is a potential function that may decay to zero
at infinity and $g: \mathbb{R}_{+} \rightarrow \mathbb{R}$ is a continuous function with subcritical growth. We deal with supercritical case $q\geq 2^*_s:=2N/(N-2s)$. Our approach is based on variational
methods combined with penalization technique and $L^{\infty}$-estimates.
\vspace{0.3cm}\\
\small{\textbf{\emph{2010 Mathematics Subject Classification:}} 35A15, 35R11, 35B33}
\vspace{0.3cm}\\
\small{\textbf{\emph{Keywords:}} Fractional magnetic operator; Vanishing potential; Supercritical exponent; Penalization technique; Moser iteration}
\end{abstract}

\maketitle

\section{Introduction}

In this work, we study the following class of fractional magnetic Schr\"{o}dinger equations
\begin{equation}\label{problem}
	(-\Delta)_{A}^{s}u+V(x)u=g(\vert u\vert^{2})u+\lambda\vert u\vert^{q-2}u, \quad \mbox{in } \mathbb{R}^{N}, \tag{$P_{\lambda,A}$}
\end{equation}
where $\lambda$ is a nonnegative parameter, $s\in (0,1)$, $N>2s$, $A :\mathbb{R}^N \rightarrow \mathbb{R}^N$ is the magnetic potential, $V:\mathbb{R}^N \rightarrow \mathbb{R}$ is a continuous and nonnegative potential, $u:\mathbb{R}^N \rightarrow \C$, $g:\mathbb{R}_{+} \rightarrow \mathbb{R}$ is a continuous function, the exponent $q\geq 2^*_s:=2N/(N-2s)$. The number $2^*_s$ is known as the fractional critical Sobolev exponent. The \textit{fractional magnetic Laplacian} $(-\Delta)_{A}^{s}$ has been defined as follows
 \begin{eqnarray}\label{1.1}
	(-\Delta)_{A}^{s}u(x):=C_{N,s}\lim_{\varepsilon \rightarrow 0}\int_{B^{c}_{\varepsilon}(x)}\frac{u(x)-u(y)e^{i A\big(\frac{x+y}{2}\big).(x-y)}}{|x-y|^{N+2s}}\mathrm{d}x, \quad C_{N,s}=\frac{4^s\Gamma(\frac{N+2s}{2})}{\pi^{N/2}|\Gamma(-s)|}.
 \end{eqnarray}
This nonlocal operator has been introduced in \cite{AveniaSq,Ichinose2} and it can be seen as a fractional extension of the magnetic pseudorelativistic operator or Weyl pseudodifferential operator with mid-point prescription. For details on the consistency of definition \eqref{1.1} and a more complete discussion on this subject, we refer the readers to \cite{Ichinose,AveniaSq,Squassina,Ichinose2}. In \cite{Squassina}, the authors have shown that when $A$ is sufficiently smooth, $(-\Delta)_{A}^{s}u$ can be viewed as a fractional counterpart of the magnetic Laplacian operator $-\Delta_{A}$ which is defined as follows
 \[
  -\Delta_{A}:=\left(\frac{1}{i} \nabla-A\right)^2=-\Delta u-\frac{2}{i}A(x) \cdot\nabla u + |A(x)|^2u-\frac{1}{i}u \textrm{div} A(x),
 \]
see \cite{avron,SULEM} for more information on this operator. Magnetic nonlinear Schr\"{o}dinger equations arise by the study of standing wave solutions for the following time-dependent Schr\"{o}dinger equation with magnetic field 
\begin{equation}\label{OPSCH}
	i \frac{\partial \Psi}{\partial t}= \left(\frac{1}{i} \nabla-A(x)\right)^2\Psi+W(x)\Psi-f(|\Psi|^2)\Psi, \quad \mbox{in \ }(x,t) \in \mathbb{R}^N\times\mathbb{R}_{+},
\end{equation}
where $W(x)$ is an electric potential, $f$ is the nonlinear coupling and $\Psi$ is the wave function representing the state of the particle, see for instance \cite{anto,SULEM,ReedSimon} for a physical background.  
A function of the form $\Psi(x,t):=u(x)e^{-i Et}$, with $E \in \mathbb{R}$, is a standing wave solution of \eqref{OPSCH} if and only if $u$ satisfies the following stationary equation
\begin{equation}\label{PA}
	-\Delta_{A}u +V(x)u=f(|u|^2)u, \quad\mbox{in \ }\mathbb{R}^N.
\end{equation} 
where $V(x)=W(x)-E$. 
For problems involving magnetic Laplacian operator we refer the readers to \cite{ AlvesFigFurt, ariole,chao,lions} and the references therein.

Regarding to the nonlocal magnetic equation \eqref{problem}, if the magnetic field $A \equiv 0$ and $s \in (0,1)$, then $(-\Delta)_{A}^{s}$ reduces to the fractional Laplacian operator $(-\Delta)^{s}$. In particular, Problem \eqref{problem} boils down to the fractional Schrödinger equation
\begin{equation}\label{fract1}
	(-\Delta)^{s}u+V(x)u=g(u)+\lambda\vert u\vert^{q-2}u, \quad \mbox{in } \mathbb{R}^{N}.
\end{equation}
The fractional Laplacian operator has been widely studied due to its vast fields of applications, such as, obstacle problems, flame propagation, minimal surfaces, conservation laws, financial
market, optimization, crystal dislocation and phase transition, see \cite{silvetre, guia,BucurVald,bisci} for more details. From the mathematical point of view, there is a huge literature related to fractional Schr\"{o}dinger equations like \eqref{fract1} under various classes of assumptions on the potential and nonlinear terms, see for instance \cite{bisci,UB,Secchi} and references therein.

We are concerned with equations involving potential that may vanishes at infinity. In \cite{alvessouto}, the authors studied existence of solutions to the local case ($s=1$) of equation \eqref{fract1} when $\lambda=0$, $g$ is subcritical and the potential $V(x)$ has a decay behavior
 \begin{equation}\label{alves}
 	\frac{1}{R^{4}}\inf_{|x|\geq R}|x|^{4}V(x) \geq\Lambda>0.	
 \end{equation} 
This work was extended in several directions, for instance, we cite \cite{plaplacian,jm1} for quasilinear problems, \cite{alves} for Choquard-type equation, \cite{jm2} for Kirchhoff-type equation, \cite{chao} for the magnetic equation \eqref{PA} involving vanishing potential and \cite{UB} for the fractional magnetic equation \eqref{fract1} involving vanishing potential. In these works, the decay \eqref{alves} was adapted to the respective class of problems. Inspired by \cite{alvessouto}, the existence of solutions is obtained by applying variational methods jointly with the penalization method in the spirit of \cite{delpinofelmer}. Although there is a relevant progress in the theory of fractional magnetic Schrödinger equation, see for instance \cite{AveniaSq, ambro3, ambro4,QKXw} and the references therein, as far as we know, nothing has been done for fractional magnetic equations involving vanishing potential.

Motivated by the above discussion and inspired by \cite{alvessouto,UB,chao}, we study the existence of solutions for Problem \eqref{problem}.
The presence of the fractional magnetic Laplacian operator brings additional difficulties. Firstly, we must consider this problem for complex valued functions and we need more delicate estimates, that can be obtained with the aid of diamagnetic inequalities, see \cite{AveniaSq,Lieb}. Secondly, in the local magnetic case, it can be applied arguments involving the following Kato’s inequality (see \cite{kato})
\begin{equation*}
	-\Delta|u|\leq {\bf \Re}\Big(sign(u)\left(-\Delta_{A} u \right)\Big),
\end{equation*}
where $\Re(z)$ denotes the real part of $z \in \C$. However, in the nonlocal magnetic framework, it is believed that a Kato’s inequality is available for $(-\Delta)^{s}_{A}$ but we are not able to prove it except for rough functions which are bounded from below and above, see \cite{ambro3}. For this reason, several arguments used in \cite{chao} are not adaptable for our case. Finally, it is worth mentioning that we are also dealing with a supercritical perturbation, which provokes more lack of compactness. In order to overcome such difficulties, we introduce two auxiliary problems to recover some compactaness and we control the parameters $\lambda$ and $\Lambda$ to relate the solution of the auxiliary problem with the original \eqref{problem}. Our approach is based on an adapted version of Penalization method jointly with $L^{\infty}$--estimates.



\vspace{0,5cm}

Throughout this work we assume that $A\in C(\mathbb{R}^{N},\mathbb{R}^{N})$. The potential $V\in C(\mathbb{R}^{N},\mathbb{R})$ satisfies the following assumptions:

\begin{itemize}
\item [$(V_1)$] $ V(x) \geq 0$,  $ \forall x \in \mathbb{R}^N$;
\item [$(V_2)$] $V(x)\leq V_{\infty}$, $\forall x \in B_1(0)$ for some constant $V_{\infty}>0$;
\item [$(V_3)$] There exist $\Lambda>0$ and $R_0 > 1$ such that
\begin{equation*}
\frac{1}{R_0^{4s}}\inf_{|x|\geq R_0}|x|^{4s}V(x) \geq\Lambda.	
\end{equation*}
\end{itemize}
The nonlinearity  $g\in C(\mathbb{R_{+}},\mathbb{R})$ satisfies the following hypotheses:
\begin{itemize}
\item [$(g_1)$] $\displaystyle\limsup_{t \rightarrow 0^+}\frac{g(t)}{t^{\frac{2^*_{s}-2}{2}}}$< $+\infty$;
\item [$(g_2)$]There exists $p \in (2, 2^*_{s})$ such that 
\begin{equation*}
	\limsup_{t \rightarrow +\infty}\frac{g(t)}{t^{\frac{p-2}{2}}}< +\infty;
\end{equation*}
\item [$(g_3)$] There exists $\theta \in (2, p]$ such that
\begin{equation*}
	0<\frac{\theta}{2} G(t)\leq tg(t), \quad \forall t>0,
\end{equation*}
where $G(t):=\int_{0}^{t}g(\tau)d\tau$.
\end{itemize}
\vspace{0,2cm}

Let
\begin{equation*}
	E:=\Bigg\{ u \in \mathcal{D}_{A}^{s,2}(\mathbb{R}^N,\C): \int_{\mathbb{R}^N}V(x)|u|^2\,\mathrm{d}x <\infty\Bigg\}.
\end{equation*}
We say that a function $u\in E$ is a weak solution of Problem \eqref{problem}, if there holds
\begin{eqnarray*}\label{S_0}
	\Re\left(\int \! \!\! \int_{\mathbb{R}^{2N} }\frac{\Big(u(x)-u(y)e^{i A\big(\frac{x+y}{2}\big).(x-y)}\Big)\overline{\Big(\phi(x)-\phi(y)e^{i A\big(\frac{x+y}{2}\big).(x-y)}\Big)}}{|x-y|^{N+2s}}\,\mathrm{d}x\mathrm{d}y\right) \nonumber\\ 
	+\,\,\Re\left(  \int_{\mathbb{R}^{N}}V(x)u\overline{\phi} \,\mathrm{d}x-
	\int_{\mathbb{R}^{N}}g(|u|^2)u\overline{\phi}\,\mathrm{d}x - \lambda\int_{\mathbb{R}^{N}}|u|^{q-2}u \overline{\phi}\,\mathrm{d}x \right)=0, \quad \forall\, \phi \in E,
\end{eqnarray*}
see Section \ref{q} for more details.

\vspace{0,2cm}

The main result of this paper can be stated as follows:

\begin{thm}\label{theorem}
	Assume that $(V_1 )$-$(V_3)$ and $(g_1)$-$(g_3)$ are satisfied. Then, there are $\lambda_0,\Lambda_0 > 0$ such that,
	for each $\lambda \in  [0, \lambda_0)$  and $\Lambda \geq \Lambda_0$, the problem \eqref{problem} has a nontrivial weak solution.
\end{thm}

\begin{remark}
	We emphasize that our main result extends and complements \cite{alvessouto,chao,UB}. Precisely, if $\lambda=0$, $A\equiv0$ and $s=1$, then \eqref{problem} boils down to the problem studied in \cite{alvessouto}. If $\lambda \neq0$ and $A \not\equiv  0$, then our main result extends \cite{chao} for the fractional magnetic setting with supercritical perturbation. Furthermore, for the case $A\not\equiv0$ it extends \cite{UB} for the fractional magnetic setting. 
\end{remark}

\begin{remark}
	An example of potential that satisfies $(V_{1})$--$(V_{3})$ is given by
		\begin{equation*}
			V(x) = \left\{
			\begin{array}{ccl}
				\varrho_{1},& \mbox{if} & |x|<R_0-\varrho_{2},\\
				|x| -R_0 + \varrho_{1} +\varrho_{2}, & \mbox{if} & R_{0}-\varrho_{2} \leq|x|< R_0,\\
						\dsp\frac{R_{0}^{4s}}{|x|^{4s}}(\varrho_{1}+\varrho_{2}), & \mbox{if} & |x|\geq R_0,
			\end{array}
			\right.
		\end{equation*}
		where $\varrho_{1}\geq 0$ and $0< \varrho_{2}<R_0$. A function $g$ that satisfies $(g_{1})$--$(g_{3})$ is given by
		\begin{equation*}
			g(t) = \left\{
			\begin{array}{ccl}
				0,& \mbox{if} & t=0,\\
				\sigma t^{\frac{2^*_{s}-2}{2}}, & \mbox{if} & 0< t< 1,\\
				\sigma t^{\frac{p-2}{2}}, & \mbox{if} & t\geq 1,
			\end{array}
			\right.
		\end{equation*}
		where $\sigma >0$.
\end{remark}

\vspace{0,3cm}

\noindent {\bf Notations.} In what follows $C$, $C_i$ denote positive constants, $B_R$ denote the open ball centered at the origin with radius $R>0$, $o_{n}(1)$ denotes a sequence which converges to $0$ as $n\rightarrow\infty$, $\Re(z)$ denotes the real part of $z \in \C$ and  $\overline{z}$ denotes its complex conjugate.

\vspace{0,3cm}

\noindent {\bf Outline.}  The remainder of this paper is organized as follows: In the forthcoming Section we introduce some preliminary results which will be useful in the remainder of the work. In Sections \ref{3} and \ref{4}, we show the existence of nontrivial solution for an auxiliary problem associated to \eqref{problem}.  Section \ref{5} is devoted to obtain a suitable $L^{\infty}$--estimate of the solution of a auxiliary problem. In the final Section \ref{6}, we prove that the solution of the auxiliary problem is in fact a solution for problem \eqref{problem}.

\section{Preliminary results}\label{q}

In this Section we collect some preliminary concepts and definitions which will be used throughout the work. Initially, we collect some facts about the fractional magnetic Sobolev space. For $s\in (0, 1)$ and magnetic field $A \in C(\mathbb{R}^N,\mathbb{R}^N)$, we define  $\mathcal{D}_{A}^{s,2}(\mathbb{R}^N,\C)$ as the completion of the set of $C_0^{\infty}(\mathbb{R}^N,\C)$ with respect to the so called  magnetic Gagliardo semi-norm 
	\begin{equation*}
		[u]_{s,A}:=\left(\frac{C_{N,s}}{2}\int\!\!\!\int_{\mathbb{R}^{2N}}\frac{|u(x)-u(y)e^{i A\big(\frac{x+y}{2} \big).(x-y)}|^2}{|x-y|^{N+2s}}\,\mathrm{d}x\mathrm{d}y\right)^{\frac{1}{2}}.
	\end{equation*}
	The space $\mathcal{D}_{A}^{s,2}(\mathbb{R}^N,\C)$ can be characterized as
	\begin{equation*}
		\mathcal{D}_{A}^{s,2}(\mathbb{R}^N,\C):= \Bigg\{ u \in L^{2^*_{s}}(\mathbb{R}^N,\C): [u]_{s,A} <\infty \Bigg \},
	\end{equation*}
	and, it is a Hilbert space with respect to the inner product
	\begin{eqnarray*}
		\langle u,v\rangle_{s,A}:= \frac{C_{N,s}}{2}\Re\left(\int \! \! \! \int_{\mathbb{R}^{2N} }\frac{\Big(u(x)-u(y)e^{i A\big(\frac{x+y}{2}\big).(x-y)}\Big)\overline{\Big(v(x)-v(y)e^{i A\big(\frac{x+y}{2}\big).(x-y)}\Big)}}{|x-y|^{N+2s}}\,\mathrm{d}x\mathrm{d}y \right).
\end{eqnarray*} 
Note that for magnetic field $A \equiv 0$, we recover the classical definition of 
$$
\mathcal{D}^{s,2}(\mathbb{R}^N, \C):= \Bigg\{ u \in L^{2^*_{s}}(\mathbb{R}^N,\C): [u]_{s} <\infty \Bigg\},
$$
where
\begin{equation*}
	[u]_{s}=\Bigg(\frac{C_{N,s}}{2}\int\!\!\!\int_{\mathbb{R}^{2N}}\frac{|u(x)-u(y)|^2}{|x-y|^{N+2s}}\,\mathrm{d}x\mathrm{d}y\Bigg)^{\frac{1}{2}}
\end{equation*}
denotes the Gagliardo semi-norm of a function $u$. For a more complete discussion on fractional Sobolev spaces, we refer the readers to \cite{guia}. Henceforth, we omit normalization constant $\frac{C_{N,s}}{2}$.
Arguing as in \cite[Lemma 3.1]{AveniaSq}, we see that the following result holds:
	\begin{lem}(Diamagnetic inequality)
		If $u \in\mathcal{D}_{A}^{s,2}(\mathbb{R}^N,\C)$, then $|u| \in \mathcal{D}^{s,2}(\mathbb{R}^N,\mathbb{R})$ and we have
		\begin{equation} \label{ime_1}
			[|u|]_s \leq [u]_{s,A}.
		\end{equation}
		We also have the following pointwise diamagnetic inequality
		\begin{equation*}
			\big||u(x)|-|u(y)|\big|\leq \Big|u(x)-u(y)e^{i A\big(\frac{x+y}{2}\big).(x-y)}\Big|.
		\end{equation*}
	\end{lem}

   Due to the presence of $V(x)$, we introduce the subspace of $\mathcal{D}_{A}^{s,2}(\mathbb{R}^N,\C)$
\begin{equation*}
E:=\Bigg\{ u \in \mathcal{D}_{A}^{s,2}(\mathbb{R}^N,\C): \int_{\mathbb{R}^N}V(x)|u|^2\,\mathrm{d}x <\infty\Bigg\}
\end{equation*}
which is a Hilbert space when endowed with the inner product,
\begin{eqnarray*}
	\langle u,v\rangle&:=&
	 \Re \left(\int \! \! \! \int_{\mathbb{R}^{2N} }\frac{\Big(u(x)-u(y)e^{i A\big(\frac{x+y}{2}\big).(x-y)}\Big)\overline{\Big(v(x)-v(y)e^{i A\big(\frac{x+y}{2}\big).(x-y)}\Big)}}{|x-y|^{N+2s}}\,\mathrm{d}x\mathrm{d}y \right) \\
	&&	+\,\, \Re \left( \int_{\mathbb{R}^N}V(x)u\overline{v}\,\mathrm{d}x \right)
\end{eqnarray*}
and its correspondent norm, $\|u\|:=\sqrt{\langle u,u \rangle}$.
 
      We recall the following embeddings of the fractional Sobolev spaces into Lebesgue spaces, see \cite[Theorem 6.5]{guia}.
\begin{lem}\label{im2.1} Let $s \in (0, 1)$ and suppose that $N > 2s$. Then, there exists a positive constant $S = S(N, s)$ such that, for any $w \in  \mathcal{D}^{s,2}(\mathbb{R}^N,\mathbb{R})$, we have
\begin{equation}\label{ime_2}
		\|w\|_{2^*_{s}}^2\leq S^{-1}[w]^2_{s}.
\end{equation}
\end{lem}
By combining \eqref{ime_1} with \eqref{ime_2}, we can deduce
\begin{equation}\label{im_1.0}
	\||u|\|_{2^*_{s}}^2\leq S^{-1}[|u|]^2_s \leq S^{-1}[u]_{s,A}^2.
\end{equation} 
Consequently, the following result holds:
\begin{lem}\label{lem2.2} The space $E$ is continuously embedded into $L^{2^*_{s}}(\mathbb{R}^N,\C)$ and compactly embedded into $ L_{loc}^{\alpha}(\mathbb{R}^N,\C)$ for any $\alpha \in (2,2^*_{s})$.
\end{lem}
  
The energy functional $\mathcal{F}_{\lambda}: E \rightarrow \mathbb{R}$ associated with Problem \eqref{problem} is given by 
\begin{equation}
	\mathcal{F}_{\lambda}(u)=\frac{1}{2}\|u\|^2-\frac{1}{2}\int_{\mathbb{R}^{N}}G(|u|^2)\,\mathrm{d}x-\frac{\lambda}{q}\int_{\mathbb{R}^{N}}|u|^q\,\mathrm{d}x.
\end{equation}
In view of $( g_1)$ and $( g_2)$ there exists $C_0 > 0$ such that
\begin{equation}\label{eq1.1}
	|t^2g(t^2)|\leq C_0|t|^{2^*_{s}} \quad \mbox{ and } \quad |t^2g(t^2)|\leq C_0|t|^{p}, \quad \forall\, t \geq 0,
\end{equation}
and by $(g_3)$, there exist $C_1$ and $C_2$ such that
\begin{eqnarray}\label{eq1.2}
	G(t^2)\geq C_1t^{\theta}-C_2, \quad \forall\, t \geq 0.
\end{eqnarray}
Hence, \eqref{eq1.1} and Lemma \ref{lem2.2} imply that $\mathcal{F}_{\lambda}$ is well defined in $E$ if and only if $q = 2^*_{s}$. Thus, we are not able to apply variational methods directly because the functional $\mathcal{F}_\lambda$ is not well defined on $E$ unless $q = 2^*_{s}$. However, to overcome this difficulty, we introduce an adequate modification on the nonlinearity $g(|u|^2)u + \lambda |u|^{q-2}u$, whose approach will be presented in the next section.

                 \section{Auxiliary problems}\label{section_3}\label{3}

In this section, in order to apply minimax methods to obtain a solution for \eqref{problem}, we consider two auxiliary problems. We start by introducing a new nonlinearity. For given $k\in \mathbb{N}$, we define the function $f_{\lambda,k}:\mathbb{R_{+}}\rightarrow \mathbb{R}$ by
\begin{equation}\label{flambdak}
f_{\lambda, k}(t) = \left\{
\begin{array}{ccl}

	g(t)+\lambda t^{\frac{q-2}{2}}, & \mbox{if} & t\leq k,\\
	g(t) +\lambda k^{\frac{q-p}{2}}t^{\frac{p-2}{2}}, & \mbox{if} & t\geq k.
\end{array}
\right.
\end{equation}
By using $(g_1)$ and  $(g_2)$, it is not hard to check that $f_{\lambda,k}$ admits the following properties:
\begin{itemize}
	\item [$(f_1)$] $|f_{\lambda,k}(t)|\leq C_0(1+\lambda k^{\frac{q-p}{2}})t^\frac{p-2}{2}$, \, $\forall \, t \geq 0$;
\item [$(f_2)$] $|f_{\lambda,k}(t)|\leq C_0(1+\lambda k^{\frac{q-p}{2}})t^\frac{2^*_{s}-2}{2}$, \, $\forall \, t \geq 0$.
\end{itemize}
Moreover, denoting $F_{\lambda,k}(t)=\int_{0}^{t}f_{\lambda,k}(\tau)\,\mathrm{d}\tau$, there holds
\begin{equation}\label{Flambdak}
F_{\lambda, k}(t) = \left\{
\begin{array}{ccl}
	G(t)+\frac{2\lambda}{q} t^{\frac{q}{2}}, & \mbox{if} & t\leq k,\\
	G(t) +\frac{2\lambda}{p} k^{\frac{q-p}{2}}t^{\frac{p}{2}} +2\lambda(\frac{1}{p}-\frac{1}{q})k^{\frac{q}{2}}, & \mbox{if} & t\geq k.
\end{array}
\right.
\end{equation}
Now, by condition $(g_3)$ and combining \eqref{flambdak} with \eqref{Flambdak}, a direct computation shows that
\begin{equation}\label{Fflambdak}
0 \leq f_{\lambda,k}(t^2)t^2-\frac{\theta}{2}F_{\lambda, k}(t^2) = \left\{
	\begin{array}{ccl}
	g(t^2)t^2-\frac{\theta}{2}G(t^2)+\Big(\frac{q-\theta}{q}\Big) \lambda t^q, & \mbox{if} & t\leq k,\\
	g(t^2)t^2-\frac{\theta}{2}G(t^2) +\lambda k^{\frac{q-p}{2}}t^p\Big(\frac{p-\theta}{p}\Big) +\theta \lambda k^{\frac{q}{2}}\Big( \frac{q-p}{qp}\Big) , & \mbox{if} & t\geq k.
	\end{array}
	\right.
\end{equation}
Using \eqref{eq1.2} and \eqref{Flambdak}, we deduce
\begin{equation}\label{Fflambdak_1}
	F_{\lambda,k}(t^2)\geq C_1|t|^\theta-C_2, \quad t \geq 0.
\end{equation}

Now, related with $f_{\lambda,k}$, we shall consider the auxiliary problem
\begin{equation}\label{aux1}
	(-\Delta)_{A}^{s}u+V(x)u=f_{\lambda,k}(\vert u\vert^{2})u, \quad \mbox{in } \mathbb{R}^{N}. \tag{$A_{\lambda,k}$}
\end{equation} 
We say that a function $u \in E$ is a weak solution of the problem \eqref{aux1}, if 
\begin{eqnarray*}\label{S_00}
	&&\Re\left(\int \! \! \! \int_{\mathbb{R}^{2N} }\frac{\Big(u(x)-u(y)e^{i A\big(\frac{x+y}{2}\big).(x-y)}\Big)\overline{\Big(\phi(x)-\phi(y)e^{i A\big(\frac{x+y}{2}\big).(x-y)}\Big)}}{|x-y|^{N+2s}}\,\mathrm{d}x\mathrm{d}y\right)\\
	&& +\, \Re\left(\int_{\mathbb{R}^{N}}V(x)u\overline{\phi} \,\mathrm{d}x  
	- \int_{\mathbb{R}^{N}}f_{\lambda,k}(|u|^2)u\overline{\phi}\,\mathrm{d}x \right)=0,\quad \forall\, \phi \in E.
\end{eqnarray*}
Note that if $u$ is a weak solution of the problem \eqref{aux1} and satisfies $|u(x)| \leq  k$ for all $x \in \mathbb{R}^N$, then  $u$ is a weak solution of the problem \eqref{problem}. The energy functional $I_{\lambda,k}:E\rightarrow\mathbb{R}$ associated with Problem \eqref{aux1} is given by
\begin{equation*}
	I_{\lambda,k}(u)=\frac{1}{2}\|u\|^2-\frac{1}{2}\int_{\mathbb{R}^{N}}F_{\lambda,k}(|u|^2)\,\mathrm{d}x.
\end{equation*}
In view of $(f_1)$, $(f_2)$ and Lemma \ref{lem2.2}, the functional $I_{\lambda,k}$ is well defined. Thus, it is easy to see that the weak solutions of \eqref{aux1} correspond to critical points of the energy functional $I_{\lambda,k}$. In order to apply Critical Point Theory we need to recover some compactness. However, we are not able to prove the Palais-Smale condition. In order to overcome this problem we use the penalization method introduced in \cite{delpinofelmer} and adapted in \cite{alvessouto,UB}. For this purpose, we introduce some definitions.
 
We fix $\nu=2\theta/(\theta-2)$, where $\theta$ is given in $(g_3)$ and we define the function $\hat{f}_{\lambda,k}: \mathbb{R}^N\times\mathbb{R_{+}}\rightarrow \mathbb{R}$ by
    \begin{equation}
\hat{f}_{\lambda, k}(x,t) = \left\{
   \begin{array}{rcl}
f_{\lambda,k}(t),& \mbox{if} & \nu f_{\lambda,k}(t)\leq V(x),\\
\dfrac{V(x)}{\nu}, & \mbox{if} & \nu f_{\lambda,k}(t) > V(x).\\
\end{array}
\right.
\end{equation}
Furthermore, considering $R_0 > 1$ given in condition $(V_3 )$, we define
\begin{equation}\label{hlambdak}
h_{\lambda, k}(x,t) = \left\{
\begin{array}{rcl}
	\!f_{\lambda,k}(t),& \mbox{if}& |x|\leq R_0,\\
	\hat{f}_{\lambda,k}(x,t),& \mbox{if}&  |x| > R_0.\\
\end{array}
\right.
\end{equation}
We introduce the second auxiliary problem
\begin{equation}\label{aux2}
	(-\Delta)_{A}^{s}u+V(x)u=h_{\lambda,k}(x,\vert u\vert^{2})u, \quad \mbox{in }\mathbb{R}^{N}, \tag{$B_{\lambda,k}$}
\end{equation}
where, $H_{\lambda,k}(x,t)=\int_{0}^{t}h_{\lambda,k}(x,\tau)\,\mathrm{d}\tau$. A direct computation shows that for all $t \geq 0$, the following inequalities hold:
\begin{eqnarray}
&h_{\lambda,k}(x,t)\leq f_{\lambda,k}(t),&\forall\, x \in \mathbb{R}^N, \label{del_1}\\
&H_{\lambda,k}(x,t)\leq F_{\lambda,k}(t),&\forall\, x \in \mathbb{R}^N,\label{del_2}\\
&\!\!\!\!h_{\lambda,k}(x,t)\leq \dsp\frac{V(x)}{\nu},& \mbox{if}\, |x|>R_0,\label{del_3}\\
&\!H_{\lambda,k}(x,t)\leq \dsp\frac{V(x)}{\nu}t,& \mbox{if}\,|x|>R_0,\label{del_4}\\
&\!\!h_{\lambda,k}(x,t)=f_{\lambda,k}(t),&\mbox{if}\, |x|\leq R_0.\label{del_4.0}
\end{eqnarray}
Moreover, combining \eqref{del_4.0} with \eqref{Fflambdak} we obtain
\begin{equation}\label{Hhlambdak}
	h_{\lambda,k}(x,t^2)t^2-\frac{\theta}{2}H_{\lambda,k}(x,t^2)\geq 0, \quad\forall\, |x|\leq R_0 \mbox{ \ and \ } \forall\, t\geq 0 .
\end{equation}
We say that a function $u \in E$ is a weak solution of the problem \eqref{aux2}, if satisfies
\begin{eqnarray}\label{S_1}
	&&\Re\left(\int \! \! \! \int_{\mathbb{R}^{2N} }\frac{\Big(u(x)-u(y)e^{i A\big(\frac{x+y}{2}\big).(x-y)}\Big)\overline{\Big(\phi(x)-\phi(y)e^{i A\big(\frac{x+y}{2}\big).(x-y)}\Big)}}{|x-y|^{N+2s}}\,\mathrm{d}x\mathrm{d}y\right)\nonumber\\ &&+\, \Re\left(\int_{\mathbb{R}^{N}}V(x)u\overline{\phi} \,\mathrm{d}x  
	- \int_{\mathbb{R}^{N}}h_{\lambda,k}(x,|u|^2)u\overline{\phi}\,\mathrm{d}x \right)=0, \quad \forall\, \phi \in E.
\end{eqnarray}
Note that if $u$ is a weak solution of the problem \eqref{aux2} and satisfies the estimate 
\begin{equation*}
	\nu f(|u(x)|^2) \leq V(x)|u(x)|^2,\quad \forall\,|x|>R_0,
\end{equation*}
then $h_{\lambda,k}(x|u|^2)u=f_{\lambda,k}(|u|^2)u$ and $u$ is indeed a solution of the problem \eqref{aux1}. The Euler–Lagrange functional associated with Problem \eqref{aux2} is given by
\begin{equation*}
	J_{\lambda,k}(u)=\frac{1}{2}\|u\|^2-\frac{1}{2}\int_{\mathbb{R}^{N}}H_{\lambda,k}(x,|u|^2)\,\mathrm{d}x.
\end{equation*}
A direct computation shows that $J_{\lambda,k}\in C^1(E,\mathbb{R})$ with
\begin{eqnarray*}\label{J}
	J'_{\lambda,k}(u)\phi &=&\Re\left(\int \! \! \! \int_{\mathbb{R}^{2N} }\frac{\Big(u(x)-u(y)e^{i A\big(\frac{x+y}{2}\big).(x-y)}\Big)\overline{\Big(\phi(x)-\phi(y)e^{i A\big(\frac{x+y}{2}\big).(x-y)}\Big)}}{|x-y|^{N+2s}}\,\mathrm{d}x\mathrm{d}y \right)\\ 
	&& +\,\Re\left(\int_{\mathbb{R}^{N}}V(x)u\overline{\phi} \,\mathrm{d}x  
	- \int_{\mathbb{R}^{N}}h_{\lambda,k}(x,|u|^2)u\overline{\phi}\,\mathrm{d}x \right), \quad \forall \,u, \phi \in E.
\end{eqnarray*}
Thus the weak solutions of \eqref{aux2} are precisely the critical points of $J_{\lambda,k}$.
\section{Existence of solutions for Problem \eqref{aux2}}\label{4}

	In this section, we will check the mountain pass geometry to the energy functional associated to problem \eqref{aux2}. We recall that a functional $I \in C^{1}(X,\mathbb{R})$ ($X$ is a  Banach space) satisfies the Palais–Smale condition at level $c \in \mathbb{R}$ ($(PS)_c-$condition for short) if any sequence  $(u_n) \subset X$ such that $I(u_n) \rightarrow c \mbox{ \ and \ } I'(u_n) \rightarrow 0$ as $n \rightarrow \infty $, has a convergent subsequence in $X$. A sequence $(u_n) \subset X$ satisfying the previous convergences is called Palais–Smale sequence for $I$ at level $c \in \mathbb{R}$ ($(PS)_c-$sequence for short). Next, we will verify the facts stated above for the energy functional  associated to the problem \eqref{aux2}. First, we show that $J_{\lambda,k}$ has the mountain pass geometry.

\begin{lem}\label{lem4.1} 
	The functional $J_{\lambda,k}$ satisfies the follwing conditions:
\begin{itemize}
\item [$(i)$] $J_{\lambda,k}(0)=0$; 
\item [$(iii)$] there exist $\delta,\rho>0$ such that $J_{\lambda,k}(v)\geq \delta$ if $\|v\|=\rho$;
\item [$(iii)$] there exists $e \in E$ such that $\|e\|>\rho$ and $J_{\lambda,k}(e)<0$.
\end{itemize}
\end{lem}
\begin{proof}
It follows directly from the definition of $J_{\lambda,k}$, that $(i)$ holds. In order to prove $(ii)$, we note
that in view of $(f_2)$, \eqref{del_2} and Sobolev continuous embedding from $E$ into $L^{2^*_{s}}(\mathbb{R}^N,\mathbb{C})$, we see that
\begin{eqnarray*}
J_{\lambda,k}(u) &\geq& \frac{1}{2}\|u\|^2-\frac{2C_0}{2^*_{s}}\big(1+\lambda k^{\frac{q-p}{2}}\big)\frac{1}{2}\int_{\mathbb{R}^N}|u|^{2^*_{s}}\,\mathrm{d}x\nonumber\\
&\geq& \frac{1}{2}\|u\|^2-C\|u\|^{2^*_{s}},
\end{eqnarray*}	
where $C:=C(\lambda,k)>0$. Therefore, since that  $2^*_{s}>2$, we choose $\rho$ small enough such that $J_{\lambda,k}(u)\geq \delta >0$ for all $u$ in $E$ with $\|u\|=\rho$, that is, $(ii)$ holds.

  in order to prove $(iii)$, fix $\varphi \in C^{\infty}_0(\mathbb{R}^N)\backslash\{0 \}$ with $supp(\varphi) \subset B_1(0)$ and note that \eqref{del_4.0} implies 
  	\begin{equation*}
  		H_{\lambda,k}(x,|\varphi|^2)=F_{\lambda,k}(|\varphi|^2), \quad \mbox{if \ } x \in supp(\varphi).
  			\end{equation*}
Thus, using \eqref{Fflambdak_1} we obtain
\begin{eqnarray*}
	J_{\lambda,k}(t\varphi)&=&\frac{t^2}{2}\|\varphi\|^2-\frac{1}{2}\int_{supp(\varphi)}F_{\lambda,k}(|t\varphi|^2)\,\mathrm{d}x\nonumber\\
	&\leq& \frac{t^2}{2}\|\varphi\|^2 - \frac{C_1}{2}t^\theta\int_{supp(\varphi)}|\varphi|^\theta \,\mathrm{d}x+\frac{C_2}{2}|supp(\varphi)|,
\end{eqnarray*}
which implies that $J_{\lambda,k}(t\varphi) \rightarrow -\infty$ as $t \rightarrow \infty$, since $\theta >2$. Finally, assertion $(iii)$ follows for $e= t\varphi$ with $t$ large enough. 
\end{proof}	
Applying a version of the Mountain Pass Theorem without the $(PS)$ condition, \cite{w}, we obtain a Palais–Smale sequence $(u_n) \subset E$ such that
\begin{equation}\label{PSc}
	J_{\lambda,k}(u_n) \rightarrow c_{\lambda,k} \mbox{ \ and \ } J'_{\lambda,k}(u_n) \rightarrow 0,
\end{equation}
where $c_{\lambda,k}$ is the mountain pass level characterized by
\begin{equation*}
	0<c_{\lambda,k}:=\inf_{\gamma \in \Gamma_{\lambda,k}}\max_{t \in [0,1]}J_{\lambda,k}(\gamma(t))
\end{equation*}
where
\begin{equation*}
	\Gamma_{\lambda,k}:=\Big\{ \gamma \in C([0, 1], E) : \gamma(0) = 0 \mbox{ \ and \ } J_{\lambda,k}(\gamma(1)) < 0 \Big\}.
\end{equation*}
Now, we introduce the functional $I_0: E\rightarrow \mathbb{R}$ given by 
	\begin{equation*}
		I_0(u)=\frac{1}{2}\int \! \! \! \int_{\mathbb{R}^{2N}}\frac{|(u(x)-u(y)e^{i A\big(\frac{x+y}{2}\big).(x-y)})|^2}{|x-y|^{N+2s}}\,\mathrm{d}x\mathrm{d}y +\frac{1}{2}\int_{\mathbb{R}^N}V(x)|u|^2\,\mathrm{d}x-\frac{1}{2}\int_{\mathbb{R}^N}G(|u|^2)\,\mathrm{d}x.
	\end{equation*}
	This functional satisfies the conditions $(i)$, $(ii)$ and $(iii)$ of Lemma \ref{lem4.1}. Hence, 
\begin{equation}\label{minimax_1}
c_0:=\inf_{\gamma\in \Gamma_0}\max_{t \in [0,1]} I_0(\gamma(t))
\end{equation}
where
\begin{equation*}
	\Gamma_0=\{\gamma \in C([0,1], E): \gamma(0)=0 \mbox{ \ and \ } I_0(\gamma(1)) < 0\}.
\end{equation*}
In view of the definition of $f_{\lambda,k}$ in \eqref{flambdak} and $h_{\lambda,k}$ in \eqref{hlambdak}, we have $J_{\lambda,k}(u) \leq I_0(u)$ for all $u \in E$. Hence, from definition of the levels $c_{\lambda,k}$ and $c_0$, we obtain $c_{\lambda,k}\leq c_0$. It is important to point out that the level $c_0$ does not depend on $\lambda$, $k$, and $R_0$.
\begin{lem}\label{lemma4.2}
If $(u_n)$ is a $(PS)_{c_{\lambda,k}}-$sequence for $J_{\lambda,k}$, then it is bounded in $E$.
\end{lem}
\begin{proof}
	Let $(u_n) \subset E$ be a $(PS)_{c_{\lambda,k}}-$sequence. Then, using \eqref{Hhlambdak}, $h_{\lambda,k} \geq 0$ and \eqref{del_4}, respectively, we reach 
\begin{eqnarray}\label{stimative_1}
	J_{\lambda,k}(u_n)-\frac{\theta}{2}J'_{\lambda,k}(u_n)u_n&=&\Big(\frac{1}{2}-\frac{1}{\theta}\Big)\|u_n\|^2+\frac{1}{\theta}\int_{\mathbb{R}^{N}}\Big[h_{\lambda,k}(x,|u_n|^2)|u_n|^2-\frac{\theta}{2} H_{\lambda,k}(x,|u_n|^2)\Big]\,\mathrm{d}x \nonumber\\
	&\geq & \Big(\frac{1}{2}-\frac{1}{\theta}\Big)\|u_n\|^2+\frac{1}{\theta}\int_{ B_{R_0}^{c}(0)}\Big[h_{\lambda,k}(x,|u_n|^2)|u_n|^2-\frac{\theta}{2} H_{\lambda,k}(x,|u_n|^2)\Big]\,\mathrm{d}x\nonumber\\
	&\geq & \Big(\frac{1}{2}-\frac{1}{\theta}\Big)\|u_n\|^2-\frac{1}{2} \int_{ B_{R_0}^{c}(0)}H_{\lambda,k}(x,|u_n|^2)\,\mathrm{d}x\nonumber\\
&\geq& \Big(\frac{1}{2}-\frac{1}{\theta}\Big)\|u_n\|^2-\frac{1}{2\nu} \int_{ B^c_{R_0}(0)}V(x)|u_n|^2\,\mathrm{d}x\nonumber\\
	&\geq& \frac{\theta -2}{4\theta}\|u_n\|^2.
\end{eqnarray}
From this, there exist $C_1$, $C_2 > 0$ such that
\begin{equation*}
	C_1+C_2 \|u_n\| \geq \frac{\theta-2}{4\theta}\|u_n\|^2, \quad \forall n \in \mathbb{N},
\end{equation*}
which finishes the proof.
\end{proof}
\begin{lem}\label{lemmaPS}
The $(PS)_{c_{\lambda,k}}-$sequence satisfies the following property:	For each $\varepsilon >0$ there exists $r=r(\varepsilon)>R_0$ verifying
\begin{equation}\label{PS}
	\limsup_{n \rightarrow +\infty} \int_{B^c_{r}(0)} \! \! \! \! \ \int_{\mathbb{R}^{N} }\frac{\big|u_n(x)-u_n(y)e^{i A\big(\frac{x+y}{2}\big).(x-y)}\big|^2}{|x-y|^{N+2s}}\,\mathrm{d}x\mathrm{d}y +\int_{B^c_{r}(0)}V(x)|u_n|^2\,\mathrm{d}x < \varepsilon.
\end{equation}
\end{lem}
\begin{proof}
Fix $r > 2R_0$ and set $\eta_{r}\in C^{\infty}(\mathbb{R}^N,\mathbb{R})$ such that $0\leq \eta_r \leq 1$, $\eta_r=0$ in $B_{\frac{r}{2}}(0)$, $\eta_r=1$ in $B^c_r(0)$ and $|\nabla \eta_r(x)|\leq \frac{C}{r}$, for some $C > 0$ independent of $r$. Since $J'_{\lambda,k}(u_n).\eta_ru_n=o_n(1)$, it follows from \eqref{del_3} that 
\begin{eqnarray}\label{PS_1}
& &\Re\left(\int\!\!\! \int_{\mathbb{R}^{2N}}\frac{\Big(u_n(x)-u_n(y)e^{i A\big(\frac{x+y}{2}\big).(x-y)}\Big)\overline{\Big(u_n(x)\eta_r(x)-u_n(y)\eta_r(y)e^{i A\big(\frac{x+y}{2}\big).(x-y)}\Big)}}{|x-y|^{N+2s}}\,\mathrm{d}x\mathrm{d}y\right) \nonumber\\
& &+\int_{\mathbb{R}^{N}}V(x)|u_n|^2\eta_r \,\mathrm{d}x  =\int_{\mathbb{R}^{N}}h_{\lambda,k}(x,|u_n|^2)|u_n|^2\eta_r \,\mathrm{d}x+o_n(1)\nonumber\\
&\leq& \frac{1}{\nu}\int_{\mathbb{R}^N}V(x)|u|^2\eta_r \,\mathrm{d}x + o_n(1).
\end{eqnarray}
Next, using $\overline{z_1+z_2}=\overline{z_1}+\overline{z_2}$ for all $z_1,z_2 \in \C$ and $\overline{e^{i t}}=e^{-i t}$ for all $t \in \mathbb{R}$, let us note that
\begin{eqnarray}\label{PS_2}	
&&\!\!\!\!\Re\Bigg(\Big(u_n(x)-u_n(y)e^{i A\big(\frac{x+y}{2}\big).(x-y)}\Big)\overline{\Big(u_n(x)\eta_r(x)-u_n(y)\eta_r(y)e^{i A\big(\frac{x+y}{2}\big).(x-y)}\Big)}\Bigg)\nonumber\\
&=&\!\!\!\!\! \Re \Bigg(\Big(u_n(x)-u_n(y)e^{i A\big(\frac{x+y}{2}\big).(x-y)}\Big)\overline{\Big(u_n(x)\eta_r(x)-u_n(y)\eta_r(x)e^{i A\big(\frac{x+y}{2}\big).(x-y)}\Big)}\nonumber\\
& & + \Big(u_n(x)-u_n(y)e^{iA\big(\frac{x+y}{2}\big).(x-y)}\Big)\overline{\Big(u_n(y)\eta_r(x)-u_n(y)\eta_r(y)e^{i A\big(\frac{x+y}{2}\big).(x-y)}\Big)}\Bigg)\nonumber\\
&=&\!\!\!\!\!\Re\Bigg(\overline{u_n(y)}e^{-i A\big(\frac{x+y}{2}\big).(x-y)}\Big(u_n(x)-u_n(y)e^{i A\big(\frac{x+y}{2}\big).(x-y)}\Big)\Big(\eta_r(x)-\eta_r(y)\Big)\Bigg)\nonumber\\
&  &+\eta_r(x)\Big|u_n(x)-u_n(y)e^{i A\big(\frac{x+y}{2}\big).(x-y)}\Big|^2.
\end{eqnarray}
Now, since $|\Re(z)|\leq |z|$ for all $z \in \C$, $|e^{it}|=1$ for all $t \in \mathbb{R}$  and $(u_n)$ is bounded in $E$, the H\"{o}lder’s inequality leads to
\begin{equation*}
	\frac{\overline{u_n(y)}e^{-i A\big(\frac{x+y}{2}\big).(x-y)}\Big(u_n(x)-u_n(y)e^{i A\big(\frac{x+y}{2}\big).(x-y)}\Big)\Big(\eta_r(x)-\eta_r(y)\Big)}{|x-y|^{N+2s}} \in L^1(\mathbb{R}^{2N},\C)
\end{equation*}
and 
\begin{eqnarray}\label{PS_3}
&&\Bigg|\Re\left(\int\!\!\!	\int_{\mathbb{R}^N}\frac{\overline{u_n(y)}e^{-i A\big(\frac{x+y}{2}\big).(x-y)}\Big(u_n(x)-u_n(y)e^{i A\big(\frac{x+y}{2}\big).(x-y)}\Big)\Big(\eta_r(x)-\eta_r(y)\Big)}{|x-y|^{N+2s}}\,\mathrm{d}x\mathrm{d}y\right)\Bigg|\nonumber\\
&\leq&\left(\int\!\!\!\int_{\mathbb{R}^{2N}}\frac{\big|u_n(x)-u_n(y)e^{i A\big(\frac{x+y}{2}\big).(x-y)}\big|^2}{|x-y|^{N+2s}}\,\mathrm{d}x\mathrm{d}y\right)^\frac{1}{2} \left(\int\!\!\!\int_{\mathbb{R}^{2N}}|\overline{u_n(y)}|^2\frac{\big|\eta_r(x)-\eta_r(y)\big|^2}{|x-y|^{N+2s}}\,\mathrm{d}x\mathrm{d}y \right)^{\frac{1}{2}}\nonumber\\
&\leq& C\left(\int\!\!\!\int_{\mathbb{R}^{2N}}|u_n(y)|^2\frac{\big|\eta_r(x)-\eta_r(y)\big|^2}{|x-y|^{N+2s}}\,\mathrm{d}x\mathrm{d}y \right)^{\frac{1}{2}}.
\end{eqnarray}
By combining \eqref{PS_1}, \eqref{PS_2} and \eqref{PS_3}, we write
\begin{eqnarray}\label{PS_5}
&&\int\!\!\!\int_{\mathbb{R}^{2N}}\frac{\eta_r(x)\big|u_n(x)-u_n(y)e^{i A\big(\frac{x+y}{2}\big).(x-y)}\big|^2}{|x-y|^{N+2s}}\,\mathrm{d}x\mathrm{d}y +\left(1-\frac{1}{\nu}\right)\int_{\mathbb{R}^{N}}V(x)|u_n|^2\eta_r \,\mathrm{d}x \nonumber\\
&\leq&-\Re\left(\int\!\!\!	\int_{\mathbb{R}^N}\frac{\overline{u_n(y)}e^{-i A\big(\frac{x+y}{2}\big).(x-y)}\Big(u_n(x)-u_n(y)e^{i A\big(\frac{x+y}{2}\big).(x-y)}\Big)\Big(\eta_r(x)-\eta_r(y)\Big)}{|x-y|^{N+2s}}\,\mathrm{d}x\mathrm{d}y \right)\nonumber\\
&\leq& C\left(\int\!\!\!\int_{\mathbb{R}^{2N}}|u_n(y)|^2\frac{\big|\eta_r(x)-\eta_r(y)\big|^2}{|x-y|^{N+2s}}\,\mathrm{d}x\mathrm{d}y\right)^{\frac{1}{2}} + o_n(1).
\end{eqnarray}
Similarly to \cite[Lemma 2.4]{ambro3} one may prove that 
\begin{eqnarray}\label{PSS}
	\lim_{r \rightarrow \infty}\limsup_{n \rightarrow \infty}\int\!\!\!\int_{\mathbb{R}^{2N}}|u_n(y)|^2\frac{\big|\eta_r(x)-\eta_r(y)\big|^2}{|x-y|^{N+2s}}\,\mathrm{d}x\mathrm{d}y=0.
\end{eqnarray}
Therefore, it follows from \eqref{PS_5} and \eqref{PSS} that \eqref{PS} holds. 
\end{proof}	
   In the next lemma, we prove the Palais–Smale condition for the functional $J_{\lambda,k}$. 

\begin{lem}\label{lemma4.3}
The functional $J_{\lambda,k}$ satisfies the 
$(PS)_{c_{\lambda,k}}-$condition. 
\end{lem}
\begin{proof}
Let $(u_n)\subset E$ be a $(PS)_{c_{\lambda,k}}-$sequence for $J_{\lambda,k}$. In view of Lemma \ref{lemma4.2}, the sequence $(u_n)$ is bounded in $E$ and, up to a subsequence, $u_n \rightharpoonup  u$ weakly in $E$. Thus,
\begin{eqnarray}\label{LemmaPS_1}
	o_n(1)&=&J'_{\lambda,k}(u_n)(u_n-u)\nonumber\\
	&=& \|u_n\|^2-\|u\|^2+o_n(1)- \dsp\int_{\mathbb{R}^N}h_{\lambda,k}(x,|u_n|^2)u_n(\overline{u_n-u})\mathrm{d}x\,\nonumber
	\\
	&=& \|u_n-u\|^2-\dsp\int_{\mathbb{R}^N}h_{\lambda,k}(x,|u_n|^2)u_n(\overline{u_n-u})\,\mathrm{d}x.
\end{eqnarray} 
In light of Sobolev’s compact embedding $E\hookrightarrow L^{\alpha}_{loc}(\mathbb{R}^N,\mathbb{C})$, for $\alpha \in (2, 2^*_{s})$, $(f_2)$, \eqref{del_1} and  H\"{o}lder’s inequality, we obtain
\begin{equation}\label{PS_6}
\int_{B_r(0)}h_{\lambda,k}(x,|u_n|^2)u_n(\overline{u_n-u})\,\mathrm{d}x=o_n(1), 
\end{equation}
for each $r>0$. Furthermore, taking $r$ sufficiently large, using $(f_2)$, \eqref{del_1} and H\"{o}lder’s inequality, we have
\begin{equation}\label{PS_7}
	\int_{B^c_r(0)}h_{\lambda,k}(x,|u_n|^2)u_n\overline{u}\,\mathrm{d}x=o_n(1).
\end{equation}
Hence, it follows from \eqref{PS_6} and \eqref{PS_7}, for $r>R_0$ sufficiently large, that
\begin{eqnarray}\label{j1}
	\int_{\mathbb{R}^N}h_{\lambda,k}(x,|u_n|^2)u_n(\overline{u_n-u})\,\mathrm{d}x&=&\int_{B_r^c(0)}h_{\lambda,k}(x,|u_n|^2)|u_n|^2\,\mathrm{d}x+o_n(1)\nonumber\\
	&\leq&\frac{1}{\nu}\int_{B^c_r(0)}V(x)|u_n|^2\,\mathrm{d}x +o_n(1).
\end{eqnarray}
Therefore, Lemma \ref{lemmaPS}, \eqref{LemmaPS_1} and \eqref{j1} imply that  
$$
\|u_n-u\|\rightarrow 0 \mbox{ \ as } n \rightarrow \infty,
$$
which finishes the proof. 
\end{proof}
In view of Lemmas \ref{lem4.1} and \ref{lemma4.3} we have the following result:
\begin{lem}\label{lemma4.4}
For each $\lambda > 0$ and $k \in \mathbb{N}$, Problem \eqref{aux2} has at least a weak solution $u_{\lambda,k} \in E$ such that $J_{\lambda,k}(u_{\lambda,k}) = c_{\lambda,k}$.
\end{lem}
\section{$L^{\infty}$--estimates}\label{5}
We start this section by proving a uniform estimate for the magnetic Gagliardo semi-norm of the solution $u_{\lambda,k}$ of problem \eqref{aux2},
obtained in Lemma \ref{lemma4.4}. 
\begin{lem}\label{lemma5.1}
Let $u_{\lambda,k}$ be the solution obtained in Lemma \ref{lemma4.4}. Then, there exists a constant $M_0$, which depends
only on $N, \theta, s, p$ \big(independent of $\lambda$, $k$ and $R_0$\big), such that  
$$
[u_{\lambda,k}]_{s,A}^2=\int \! \! \! \int_{\mathbb{R}^{2N} }\frac{|u_{\lambda,k}(x)-u_{\lambda,k}(y)e^{i A\big(\frac{x+y}{2}\big).(x-y)}|^2}{|x-y|^{N+2s}}\,\mathrm{d}x\mathrm{d}y \leq M_0.
$$
\end{lem}
\begin{proof} 
	In view of the estimate \eqref{stimative_1} obtained in the proof of Lemma \ref{lemma4.2} and recalling that $c_{\lambda,k}\leq c_0$, we have
	\begin{eqnarray}
		c_0\geq c_{\lambda,k}=J_{\lambda,k}(u_{\lambda,k})-\frac{\theta}{2}J'_{\lambda,k}(u_{\lambda,k}).u_{\lambda,k} \geq \frac{\theta-2}{4\theta}\|u_{\lambda,k}\|^2
	\end{eqnarray}
which implies
\begin{equation*}
	\|u_{\lambda,k}\|^2 \leq \Big(\frac{4\theta}{\theta-2}\Big)c_0=:M_0.
\end{equation*}
Therefore,
\begin{equation*}
	[u_{\lambda,k}]^2_{s,A} \leq M_0,
\end{equation*}
this completes the proof.
\end{proof}
The next Lemma is crucial in our arguments, since it establishes an important estimate involving the $L^{\infty}-$norm of the solution $u_{\lambda,k}$. For this purpose, we shall use the Moser iteration method. 

\begin{lem}\label{lemma5.2} For each $\lambda> 0$ and $k \in \mathbb{N}$, $|u_{\lambda, k}| \in  L^{\infty}(\mathbb{R}^N,\mathbb{R})$, and there exists $C > 0$ that depends only on
$N$, $s$, $\theta$, $p$ and $C_0$, such that
\begin{equation}\label{Eq Lemm5.2}
	\big\||u_{\lambda,k} |\big\|_{\infty} \leq C\big(1+\lambda k^{\frac{q-p}{2}}\big)^{\gamma} \|u_{\lambda,k}\|_{2_s^*},
\end{equation}
where $\gamma ={\frac{1}{2(\beta_1-1)}}$ and  $\beta_1=\frac{{2_s^*-p+2}}{2}$.	
\end{lem}
\begin{proof}
The proof follows some ideas from \cite[Lemma 2.8]{ambro3}. We denote, by simplicity, $u\!=\!u_{\lambda,k}$. For $L > 0$, we define $u_L=\min \left\lbrace |u|,L \right\rbrace$. Taking $\phi =u u_L^{2(\beta -1)}$ as test function in \eqref{S_1}, where $\beta >1$ will be chosen later, we can deduce
\begin{eqnarray}\label{M_1}
&&\Re\left(\int \! \! \! \int_{\mathbb{R}^{2N} }\frac{\Big(u(x)-u(y)e^{i A\big(\frac{x+y}{2}\big).(x-y)}\Big)\overline{\Big(u u_L^{2(\beta -1)}(x)-u u_L^{2(\beta -1)}(y)e^{i A\big(\frac{x+y}{2}\big).(x-y)}\Big)}}{|x-y|^{N+2s}}\,\mathrm{d}x\mathrm{d}y\right)\nonumber \\
& & +\int_{\mathbb{R}^{N}}V(x)|u|^2u_L^{2(\beta -1)} \,\mathrm{d}x- \int_{\mathbb{R}^{N}}h_{\lambda,k}(x,|u|^2)|u|^2 u_L^{2(\beta -1)}\,\mathrm{d}x=0.
\end{eqnarray}
Now, using the fact that $\Re(z) \leq |z|$, for all $z$ in $\mathbb{C}$ and $|e^{i t}|=1$ for all $t$ in $\mathbb{R}$, we obtain
\begin{eqnarray*}\label{M_2}
& & \Re\left( \Big(u(x)-u(y)e^{i A\big(\frac{x+y}{2}\big).(x-y)}\Big)\overline{\Big(u u_L^{2(\beta -1)}(x)-u u_L^{2(\beta -1)}(y)e^{i A\big(\frac{x+y}{2}\big).(x-y)}\Big)}  \right) \nonumber \\
& & \geq\Big(|u(x)|-|u(y)|\Big)\Big(|u(x)| u_L^{2(\beta -1)}(x)-|u(y)| u_L^{2(\beta -1)}(y)\Big),
\end{eqnarray*}
which implies that
\begin{eqnarray}\label{M_4}
& & \Re\left(\int \! \! \! \int_{\mathbb{R}^{2N} }\frac{\Big(u(x)-u(y)e^{i A\big(\frac{x+y}{2}\big).(x-y)}\Big)\overline{\Big(u u_L^{2(\beta -1)}(x)-u u_L^{2(\beta -1)}(y)e^{i A\big(\frac{x+y}{2}\big).(x-y)}\Big)}}{|x-y|^{N+2s}}\,\,\mathrm{d}x\mathrm{d}y\right) \nonumber \\	
&& \geq \int \! \! \! \int_{\mathbb{R}^{2N} }\frac{\Big(|u(x)|-|u(y)|\Big)\Big(|u(x)| u_L^{2(\beta -1)}(x)-|u(y)|u_L^{2(\beta -1)}(y)\Big)}{|x-y|^{N+2s}}\,\mathrm{d}x\mathrm{d}y.
\end{eqnarray}
For any $t \geq0$, let us define the functions
$$
\alpha(t):=\alpha_{L,\beta}(t)=tt_L^{2(\beta -1)},\, \mbox{where} \ \ t_L=\min\{|t|,L\},
$$
$$
\Lambda(t)=\frac{|t|^2}{2} \quad \mbox{and} \quad \Gamma(t)=\int_{0}^{t}(\alpha'(s))^{\frac{1}{2}}\,\mathrm{d}s.
$$
The following estimates hold:
\begin{equation}\label{af1}
	\Lambda'(a-b)(\alpha(a)-\alpha(b))\geq |\Gamma(a)-\Gamma(b)|^2, \quad \forall \, a,b \in \mathbb{R}
\end{equation}
and
\begin{equation}\label{af2}
\Gamma(|t|)\geq \frac{1}{\beta}|t|t_L^{\beta -1}.
\end{equation}
By using \eqref{af1}, we have
\begin{equation}\label{M_5}
|\Gamma(|u(x)|)-\Gamma(|u(y)|)|^2 \leq (|u(x)|-|u(y)|)\Big(|u(x)|u_L^{2(\beta -1)}(x) - |u(y)|u_L^{2(\beta -1)}(y) \Big).
\end{equation}
Next, combining \eqref{M_4} with \eqref{M_5}, we obtain
\begin{eqnarray}\label{M_6}
&&\Re\left(\int \! \! \! \int_{\mathbb{R}^{2N} }\frac{\Big(u(x)-u(y)e^{i A\big(\frac{x+y}{2}\big).(x-y)}\Big)\overline{\Big(u u_L^{2(\beta -1)}(x)-u u_L^{2(\beta -1)}(y)e^{i A\big(\frac{x+y}{2}\big).(x-y)}\Big)}}{|x-y|^{N+2s}}\,\mathrm{d}x\mathrm{d}y\right)\nonumber \\	
& & \geq \int \! \! \! \int_{\mathbb{R}^{2N} }\frac{|\Gamma(|u(x)|)-\Gamma(|u(y)|)|^2}{|x-y|^{N+2s}}\,\mathrm{d}x\mathrm{d}y=\big[\Gamma(|u|) \big]_{s}^2,
\end{eqnarray}
and using Lemma \ref{ime_1} and \eqref{af2}, we note
\begin{eqnarray}\label{M_7}
	[\Gamma(|u|)]_{s}^2\geq S\|\Gamma(|u|) \|^2_{2^*_{s}} \geq S\frac{1}{\beta^2}\||u|u_L^{\beta -1} \|^2_{2^*_{s}}.
\end{eqnarray}
In view of \eqref{M_1},\eqref{M_4} and \eqref{M_7}, we infer
\begin{equation*}\label{M_8}
	S\frac{1}{\beta^2}\big\||u|u_L^{\beta -1} \big\|^2_{2^*_{s}}+\int_{\mathbb{R}^{N}}V(x)|u|^2u_L^{(2(\beta -1)}\,\mathrm{d}x\leq\int_{\mathbb{R}^{N} }h_{\lambda,k}(x,|u|^2)|u|^2u_L^{2(\beta -1)}\,\mathrm{d}x.
\end{equation*}
Since $V\geq 0$, $|u|\geq u_L>0$, using the above estimate jointly with \eqref{del_1} and $(f_1)$, we reach
\begin{eqnarray}
		S\frac{1}{\beta^2}\big\||u|u_L^{\beta -1} \|^2_{2^*_{s}}\leq C_0(1+\lambda k^{\frac{q-p}{2}})\int_{\mathbb{R}^{N}}|u|^{2\beta}|u|^{p-2}\,\mathrm{d}x.
\end{eqnarray}
By H\"{o}lder’s inequality with exponents $2^*_{s}/(p-2)$ and $2^*_{s}/(2^*_{s}-p+2)$, we see
\begin{eqnarray}\label{M_9}
\big\||u|u_L^{\beta -1} \big\|^2_{2^*_{s}} &\leq& S^{-1}\beta^2C_0(1+\lambda k^{\frac{q-p}{2}})\Bigg(\int_{\mathbb{R}^{N}}|u|^{2^*_{s}}\,\mathrm{d}x\Bigg)^{\frac{p-2}{2^*_{s}}}\Bigg(\int_{\mathbb{R}^{N} }|u|^{\frac{2\beta 2^*_{s}}{2^*_{s}-p+2}}\,\mathrm{d}x\Bigg)^{\frac{2^*_{s}-p+2}{2^*_{s}}}\nonumber\\
&=& S^{-1}\beta^2C_0(1+\lambda k^{\frac{q-p}{2}})\big\||u|\big\|^{p-2}_{2^*_{s}}\Bigg(\int_{\mathbb{R}^{N} }|u|^{\frac{2\beta 2^*_{s}}{2^*_{s}-p+2}}\,\mathrm{d}x\Bigg)^{\frac{2^*_{s}-p+2}{2^*_{s}}}.
\end{eqnarray}
Moreover, recalling \eqref{im_1.0} and Lemma \ref{lemma5.1}, we have
\begin{equation}\label{eq5.15}
	\big\||u|\big\|^{2}_{2^*_{s}}\leq  S^{-1}M_0.
\end{equation}
Now, we observe that if $|u|\in L^{\frac{2^*_{s}2\beta}{2^*_{s}-p+2}}(\mathbb{R}^N,\mathbb{R})$ and combining \eqref{eq5.15} with \eqref{M_9}, then
\begin{equation}\label{M_10}
\big\||u|u_L^{\beta -1} \big\|^2_{2^*_{s}}\leq C_1(1+\lambda k^{\frac{q-p}{2}})\beta^2\Bigg(\int_{\mathbb{R}^{N} }|u|^{\frac{2\beta 2^*_{s}}{2^*_{s}-p+2}}\,\mathrm{d}x\Bigg)^{\frac{2^*_{s}-p+2}{2^*_{s}}} < \infty,
\end{equation} 
where $C_1=C_1(N,s,\theta,p,C_0)$. Hence, since $ u_L \rightarrow |u|$  almost everywhere as $L\rightarrow \infty$, using Fatou’s Lemma in \eqref{M_10}, we conclude that
\begin{equation*}
	\big\||u|^{\beta} \big\|^2_{2^*_{s}}\leq C_1(1+\lambda k^{\frac{q-p}{2}})\beta^2\Bigg(\int_{\mathbb{R}^{N} }|u|^{\frac{2\beta 2^*_{s}}{2^*_{s}-p+2}}\,\mathrm{d}x\Bigg)^{\frac{2^*_{s}-p+2}{2^*_{s}}}
	\end{equation*}
from which we deduce that
\begin{equation*}
	\big\||u|^{\beta} \big\|^{\frac{1}{\beta}}_{L^{2^*_{s}}(\mathbb{R}^N,\mathbb{R})}\leq [C_1(1+\lambda k^{\frac{q-p}{2}})]^{\frac{1}{2\beta}}\beta^{\frac{1}{\beta}}\Bigg(\int_{\mathbb{R}^{N} }|u|^{\frac{2\beta 2^*_{s}}{2^*_{s}-p+2}}\,\mathrm{d}x\Bigg)^{\frac{2^*_{s}-p+2}{2^*_{s}2\beta}}\!\!,
\end{equation*}
that is,
\begin{equation}\label{I_1}
	\big\||u| \big\|_{2^*_{s}\beta}\leq [C_1(1+\lambda k^{\frac{q-p}{2}})]^{\frac{1}{2\beta}}\beta^{\frac{1}{\beta}}\Bigg(\int_{\mathbb{R}^{N} }|u|^{\frac{2\beta 2^*_{s}}{2^*_{s}-p+2}}\,\mathrm{d}x\Bigg)^{\frac{2^*_{s}-p+2}{2^*_{s}2\beta}}
\end{equation}
thus, $|u| \in L^{2^*_{s}\beta}(\mathbb{R}^N,\mathbb{R})$.

   
Let us use inequality \eqref{I_1} in order to obtain the desired $L^{\infty}-$estimate, through Moser iteration method. For this, setting $\beta=\beta_1:=(2^*_{s}-p+2)/2$ in \eqref{I_1}, there holds
\begin{equation}\label{I_2}
\big\||u| \big\|_{2^*_{s}\beta_1}\leq [C_1(1+\lambda k^{\frac{q-p}{2}})]^{\frac{1}{2\beta_1}}\beta_1^{\frac{1}{\beta_1}}\big\||u| \big\|_{2^*_{s}}.  
\end{equation}
Now, when $\beta =\beta_2:=\beta_1^2$ in \eqref{I_1}, we have $2\beta_2 2^*_{s}/(2^*_{s}-p+2)=2^*_{s}\beta_1$ and we deduce that
\begin{eqnarray*}
	\big\||u| \big\|_{2^*_{s}\beta_2}&\leq& [C_1(1+\lambda k^{\frac{q-p}{2}})]^{\frac{1}{2\beta_2}}\beta_2^{\frac{1}{\beta_2}}\big\||u| \big\|_{2^*_{s}\beta_1}\\
	&\leq& [C_1(1+\lambda k^{\frac{q-p}{2}})]^{\frac{1}{2\beta_1}+\frac{1}{2\beta_2}}\beta_1^{\frac{1}{\beta_1}}\beta_2^{\frac{1}{\beta_2}}\big\||u| \big\|_{2^*_{s}},
\end{eqnarray*}
once the estimate \eqref{I_2} holds.
Arguing by iteration $m$ times for $m \geq 2$, with $\beta\!\!=\!\!\beta_m\!\!:=\!\!\beta_{m-1}\beta_1\!\!=\!\!\beta_1^{m}$ in \eqref{I_1} and using that $2\beta_m2^*_{s}/(2^*_{s}-p+2)=2^*_{s}\beta_{m-1}$, we deduce
\begin{equation}\label{I_3}
\big\||u| \big\|_{2^*_{s}\beta_m}\leq [C_1(1+\lambda k^{\frac{q-p}{2}})]^{\frac{1}{2\beta_1}+\frac{1}{2\beta_2}+\cdots +\frac{1}{2\beta_m}}\beta_1^{\frac{1}{\beta_1}}\beta_2^{\frac{1}{\beta_2}}\cdots \beta_m^{\frac{1}{\beta_m}}\big\||u| \big\|_{2^*_{s}},
\end{equation}
which implies that $|u|\in L^{2^*_{s}\beta_m}(\mathbb{R}^N,\mathbb{R})$, for all $m \geq 2$.
Once that,
\begin{equation*}
	\frac{1}{2}\sum_{j=1}^{m}\frac{1}{\beta_j}\leq\frac{1}{2}\sum_{j=1}^{\infty}\Big(\frac{1}{\beta_1}\Big)^j=\frac{1}{2(\beta_1-1)}=\gamma
\end{equation*}
and
\begin{equation*}
\beta_1^{\frac{1}{\beta_1}}\beta_2^{\frac{1}{\beta_2}}\cdots \beta_m^{\frac{1}{\beta_m}}\leq \beta_1^{\sum_{j=1}^{\infty}\frac{j}{\beta_1^{j}}}=\beta_1^{\tilde{\gamma}}, \quad \tilde{\gamma}=\frac{\beta_1}{(\beta_1-1)^2}
\end{equation*}
it follows from \eqref{I_3} that
\begin{equation*}
\big\||u| \big\|_{p}\leq [C_1(1+\lambda k^{\frac{q-p}{2}})]^{\gamma}\beta_1^{\tilde{\gamma}}\big\||u| \big\|_{2^*_{s}}, 
\end{equation*}
where $p:=2^*_{s}\beta_m=2^*_{s}\beta^m \geq 2^*_{s}$, for all  $m \geq 2$ . Since $\beta_m=\beta_1^{m}\rightarrow \infty$ as $m \rightarrow \infty$,
we conclude that Lemma is valid for 
$
C=C_1^{\gamma}\beta_1^{\tilde{\gamma}}
$
with 
$
\beta_1=\frac{{2_s^*-p+2}}{2}.
$
\end{proof}
 In view of Lemma \ref{lemma5.2}, since  $\eqref{Eq Lemm5.2}$  and  \eqref{eq5.15} hold, we are able to find suitable values of $\lambda$ and $k$ such that the following inequality holds true
		\begin{equation*}
			\big\||u_{\lambda,k}|\big\|_{\infty} \leq C_1^{\gamma}\beta_1^{\tilde{\gamma}}\big(1+\lambda k^{\frac{q-p}{2}}\big)^{\gamma}C_2 < k, 
		\end{equation*}
		where $C_2=(S^{-1}M_0)^{\frac{1}{2}}$. In fact, we shall verify that
		\begin{equation*}
			C_1^{\gamma}\beta_1^{\tilde{\gamma}}\big(1+\lambda k^{\frac{q-p}{2}}\big)^{\gamma}C_2 < k, 
		\end{equation*}
		or equivalently,
		\begin{equation*}
			\lambda k^{\frac{q-p}{2}}\leq \frac{1}{C_1\beta_{1}^{\frac{\tilde{\gamma}}{\gamma}}}(C_2^{-1}k)^{\frac{1}{\gamma}}-1.
		\end{equation*}
		Consider $k>0$ such that
		\begin{equation*}
			\frac{1}{C_1\beta_{1}^{\frac{\tilde{\gamma}}{\gamma}}}(C_2^{-1}k)^{\frac{1}{\gamma}}-1>0
		\end{equation*}
		and fix $\lambda_0>0$ such that
		\begin{equation*}
			\lambda <\lambda_{0} \leq \Bigg(\frac{1}{C_1\beta_{1}^{\frac{\tilde{\gamma}}{\gamma}}}(C_2^{-1}k)^{\frac{1}{\gamma}}-1\Bigg)\frac{1}{k^{\frac{q-p}{2}}}.
		\end{equation*}
		Thus, by taking $k_0>C_3:= C_1^{\gamma}\beta_1^{\tilde{\gamma}}C_2$ we obtain $\lambda_0>0$ such that
		\begin{equation}\label{eq}
			\big\||u_{\lambda,k_0}|\big\|_{\infty} \leq k_0,\quad \forall \, \lambda \in [0,\lambda_{0}).
		\end{equation}

        
\section{Proof of Theorem \ref{theorem}}\label{6}

In light of Lemma \ref{lemma4.4}, for each $\lambda >0$ and $k \in \N$, the auxilary problem \eqref{aux2} admits a solution $u_{\lambda,k}$ in $E$. Thereby, in order to prove the existence of solution for the original problem \eqref{problem}, in view that \eqref{eq} holds, it is sufficient to prove that the following inequality holds:
\begin{equation*}
	 f_{\lambda,k_0}(|u_{\lambda,k_0}(x)|^2) \leq \frac{V(x)}{\nu},\quad \forall \, |x|>R_0 \mbox{ \ and \ } \forall\, \lambda \in [0, \lambda_0).
\end{equation*}

\begin{lem}\label{lem5.3}
	For each $\lambda> 0$ and $k \in \mathbb{N}$, set $u_{\lambda,k}$ solutions for the auxiliary problem \eqref{aux2}, such that $J_{\lambda,k}(u_{\lambda,k})=c_{\lambda,k}$. Then,
	$$
	|u_{\lambda ,k}| \leq \frac{R_0^{N-2s}}{|x|^{N-2s}}\big \||u_{\lambda,k}|\big\|_\infty, \quad \forall\,  |x|\geq R_0.
	$$
\end{lem}
\begin{proof}
For the sake of simplicity, we denote $u = u_{\lambda, k}$.	Let $v$ be the $C^{\infty}(\mathbb{R}^N\backslash \{0\},\mathbb{R})$ function
\begin{equation}\label{Harmonic}
	v(x)=\frac{R_0^{N-2s}\big\||u|\big\|_{\infty}}{|x|^{N-2s}}, \quad x\not =0.
\end{equation}
Since $1/|x|^{N-2s}$ is $s$-harmonic (see for instance \cite{BucurVald}), it follows that $(-\Delta)^{s}v(x)=0$ in $\mathbb{R}^N\backslash \{0\}$. Note that
$$
|u| \leq \big \||u|\big\|_{\infty} \leq \frac{R_0^{N-2s}}{|x|^{N-2s}} \big \||u|\big\|_{\infty}, \quad \forall\,  |x|\leq R_0.
$$
Let us introduce the function $w\in \mathcal{D}^{s,2}(\mathbb{R}^N, \mathbb{R})$ defined by
$$
w(x) = \left\{
\begin{array}{ccl}
	(|u|-v)^+(x),& \mbox{if} & |x|\geq R_0,\\
	0, & \mbox{if} & |x|\leq R_0.\\
\end{array}
\right.
$$

\vspace{0,2cm}

It is worth mentioning that at this moment we could think of arguing the prove as in \cite{chao} and apply Kato’s inequality at function $\psi:=\frac{u}{|u|}w$. However, as pointed out in \cite{ambro3}, we are not able to use $\psi$ as test function and we do not have a Kato’s inequality for the fractional magnetic. Thus the arguments in \cite{chao} collapse and we are not able to use directly in our situation.  We overcome this difficulty arguing similarly to \cite{ambro3}, by introducing the function 
	\[
	 \psi_{\delta}:=\dsp\frac{u}{u_{\delta}}w, \quad \mbox{where} \hspace{0,3cm}  u_{\delta}=\sqrt{|u|^2 +\delta^2}, \hspace{0,3cm} \delta>0.
	\] 
   Thus, applying $\psi_{\delta}$ as test function in \eqref{S_1} and taking the limit when $\delta \rightarrow  0$, we use Dominated Convergence Theorem to obtain
	\begin{equation}\label{Ftest_10}
		\int \! \! \! \int_{\mathbb{R}^{2N}}\frac{\big(|u(x)|-|u(y)|\big)\big(w(x)- w(y)\big)}{|x-y|^{N+2s}}\,\mathrm{d}x\mathrm{d}y \leq \int_{\mathbb{R}^N}\left(-V(x)+h_{\lambda,k}(x,|u|^2)\right)|u|w\,\mathrm{d}x.
	\end{equation}
Now, using the fact that $v$ is $s$-harmonic and $\big[(u-v)(x)-(u-v)(y)\big]\big(|w(x)-w(y)\big) \geq \big|w(x)-w(y)\big|^2$ for all $x,y \in \mathbb{R}^N$, we can obtain
\begin{eqnarray}\label{Ftest_11}
\int \! \! \! \int_{\mathbb{R}^{2N}}\frac{\big(|u(x)|-|u(y)|\big)\big(w(x)- w(y)\big)}{|x-y|^{N+2s}}\,\mathrm{d}x\mathrm{d}y \!\!\!&=&\!\!\!\!\int \! \! \! \int_{\mathbb{R}^{2N}}\frac{\big((u-v)(x)-(u-v)(y)\big)\big(w(x)- w(y)\big)}{|x-y|^{N+2s}}\,\mathrm{d}x\mathrm{d}y \nonumber\\ 
&\geq& \!\!\!\! \int  \! \! \! \int_{\mathbb{R}^{2N}}\frac{\big(w(x)- w(y)\big)^2}{|x-y|^{N+2s}}\,\mathrm{d}x\mathrm{d}y.
\end{eqnarray}
Now, if  we write $\mathbb{R}^N=( B^c_{R_0}\cap \Theta) \cup (B^c_{R_0} \cap \Theta^{c}) \cup B_{R_0}$ where $\Theta:=\{x \in \mathbb{R}^N:|u(x)|\geq v(x)\}$
and since $w=0$ in $B^c_{R_0}\cap \Theta^c$ and $w=0$ in $B_{R_0}$, then we deduce
\begin{equation}\label{Ftest_12}
 \int_{\mathbb{R}^N}\big(-V(x)+h_{\lambda,k}(x,|u|^2)\big)|u|w\,\mathrm{d}x=\int_{B^c_{R_0}\cap \Theta}\big(-V(x)+h_{\lambda,k}(x,|u|^2)\big)|u|w\,\mathrm{d}x.
\end{equation}
Finally, by \eqref{Ftest_11}, \eqref{Ftest_12} and once that $h_{\lambda,k}(x,|u|^2)\leq \frac{1}{\nu}V(x)$, \eqref{Ftest_10} becomes
\begin{equation*}
\int \! \! \! \int_{\mathbb{R}^{2N}}\frac{\big|w(x)- w(y)\big|^2}{|x-y|^{N+2s}}\,\mathrm{d}x\mathrm{d}y \leq  \Big(\frac{1}{\nu}-1\Big)\int_{B^c_{R_0}\cap \Theta}V(x)|u|w\,\mathrm{d}x \leq 0,
\end{equation*}
showing that $w\equiv 0$. Therefore, $|u| \leq v$ in $|x|\geq R_0$ and this finishes the proof.
\end{proof}

\vspace{0,3cm}

\!\!\!\!\!\!\!\textcolor{black}{{\bf Proof of Theorem}} \ref{theorem}: Using the definition of $f_{\lambda,k}$, $(f_2)$ and combining Lemma \ref{lem5.3} with \eqref{eq}-$(V_3)$, it follows that
\begin{eqnarray}\label{ineq_4}
	\nu f_{\lambda,k_0}(|u_{\lambda,k_0}(x)|^2)&\leq& \nu f_{\lambda_0,k_0}(|u_{\lambda,k_0}(x)|^2)\nonumber \\
	&\leq&\nu \big(1+\lambda_0k_0^{\frac{q-p}{2}}\big)|u_{\lambda,k_0}(x)|^{2^*_s-2}  \nonumber\\
	&\leq& \nu \big(1+\lambda_0k_0^{\frac{q-p}{2}}\big)\Big( \frac{R_0^{N-2s}\||u_{\lambda,k_0}|\|_{\infty}}{|x|^{N-2s}}\Big)^{2^*_s-2}  \nonumber\\
	&\leq& \nu \big(1+\lambda_0k_0^{\frac{q-p}{2}}\big)\frac{R_0^{4s}}{|x|^{4s}}\||u_{\lambda,k_0}|\|_{\infty}^{\frac{4s}{N-2s}}\nonumber\\
	&\leq& \nu \big(1+\lambda_0k_0^{\frac{q-p}{2}}\big)\frac{V(x)}{\Lambda}k_0^{\frac{4s}{N-2s}}, \quad \forall \, |x|>R_0.
\end{eqnarray} 
Now, if $\Lambda\geq \Lambda_0:=\nu\big(1+\lambda_0k_0^{\frac{q-p}{2}}\big)$ and $\lambda$ in $[0,\lambda_0)$, then \eqref{ineq_4} implies that
\begin{equation*}
		\nu f_{\lambda,k_0}(|u_{\lambda,k_0}(x)|^2) \leq V(x),\quad \forall \, |x|>R_{0} \mbox{ \ and \ } \forall\, \lambda \in [0, \lambda_0).
\end{equation*}
Consequently, by \eqref{flambdak} and \eqref{hlambdak}, we deduce 
\begin{eqnarray}\label{ineq_5}
	h_{\lambda,k_0}(x,|u_{\lambda,k_0}(x)|^2)&=&\hat{f}_{\lambda,k_0}(|u_{\lambda,k_0}(x)|^2)\nonumber\\
	&=&f_{\lambda,k_0}(|u_{\lambda,k_0}(x)|^2)\nonumber\\
	&=&g(u_{\lambda,k_0}(x)) + \lambda |u_{\lambda,k_0}(x)|^{q-2}, \quad \mbox{a.e. in } \mathbb{R}^N,
\end{eqnarray}
for all $\lambda$ in $[0,\lambda_0)$ and $\Lambda \geq \Lambda_0$. Finally, by \eqref{ineq_5} and since $u_{\lambda,k_0}$ is a critical point of $J_{\lambda,k_0}$, we reach
\begin{eqnarray*}
0&=&\Re\left(\int \!\!\!\int_{\mathbb{R}^{2N} }\frac{\Big(u_{\lambda,k_0}(x)-u_{\lambda,k_0}(y)e^{i A\big(\frac{x+y}{2}\big).(x-y)}\Big)\overline{\Big(\phi(x)-\phi(y)e^{i A\big(\frac{x+y}{2}\big).(x-y)}\Big)}}{|x-y|^{N+2s}}\,\mathrm{d}x\mathrm{d}y\right) \nonumber\\
&&+\,\,\Re\left( \int_{\mathbb{R}^{N}}V(x)u_{\lambda,k_0}\overline{\phi}\,\mathrm{d}x -\int_{\mathbb{R}^{N}}h_{\lambda,k}(x,|u_{\lambda,k_0}|^2)u_{\lambda,k_0}\overline{\phi}\,\,\mathrm{d}x \right)\\
& =& \Re\left( \int_{\mathbb{R}^{N}}V(x)u_{\lambda,k_0}\overline{\phi}\,\mathrm{d}x-\int_{\mathbb{R}^{N}}g(|u_{\lambda,k_0}|^2)u_{\lambda,k_0}\overline{\phi}\,\mathrm{d}x - \lambda \int_{\mathbb{R}^N}|u_{\lambda,k_0}|^{q-2}u_{\lambda,k_0}\overline{\phi}\,\mathrm{d}x\right), 
\end{eqnarray*}
for all $\phi \in E$. Therefore, we concluded that $u_{\lambda,k_0}$ is a solution of problem \eqref{problem} for $\lambda$ in $[0, \lambda_0)$ and $\Lambda \geq \Lambda_0$. This finishes the proof of Theorem \ref{theorem}.$\blacksquare$\\

\begin{remark}\label{regular}
	Let us denote $u = u_{\lambda,k}$ the solution of \eqref{problem}. Note that if $V \in L^{\infty}(\mathbb{R}^N,\mathbb{R})$, then
	$$
	f(|u|^2)u-V(x)u+\lambda |u|^{q-2}u \in L^{\infty}(\mathbb{R}^N,\C),
	$$
	in view $|u| \in L^{\infty}(\mathbb{R}^N,\mathbb{R})$, obtained by Moser iteration method. 
	Therefore, in view of the regularity results established for the fractional Laplacian
	(see \cite{bisci} or \cite{silvetre} ), one can think that, under suitable regularity assumptions on $A$, it is
	possible to obtain more regularity on solutions to \eqref{problem}. Next, from the arguments in \cite[Remark 2.1]{ambro4} and using fractional regularity theory, we prove that $u \in C^{0,\alpha}(\mathbb{R}^N,\C)$. For instance, assume that $A \in
			L^{\infty}(\mathbb{R}^N,\mathbb{R}^N)$ and $s \in (0,\frac{1}{2})$. We set $u := v+i w$, with $v,w$ real valued. Assume that $u$ solve \eqref{problem}, that is,
	$u\in E$ and satisfies the following equation
	\begin{equation}\label{Delta_2}
		(-\Delta)_{A}^{s}u+V(x)u =g(|u|^2)u +\lambda |u|^{q-2}u, \quad \mbox{\ in \ } \mathbb{R}^N.
	\end{equation}
Thus, we may deduce that $v$ and $w$ solve, respectively,
	\begin{equation}\label{Delta_0}
		(-\Delta)^{s}v+V(x)v=g(|u|^2)v+\lambda|u|^{q-2}v-C_{A}(u,v)
	\end{equation}
	and
	\begin{equation}\label{Delta_00}
		(-\Delta)^{s}w+V(x)w=g(|u|^2)w+\lambda|u|^{q-2}w-D_{A}(u,w)
	\end{equation}
	where
	\begin{equation*}
		C_A(u,v)(x):=\int_{\mathbb{R}^N}\frac{v(y)\big[1-\cos\big(A(\frac{x+y}{2}\big)\cdot(x-y)\big)\big]+w(y)\sin\big (A(\frac{x+y}{2}).(x-y)\big)}{|x-y|^{N+2s}}\mathrm{d}y
	\end{equation*}
	and
	\begin{equation*}
		D_A(u,w)(x):=\int_{\mathbb{R}^N}\frac{w(y)\big[1-\cos\big(A(\frac{x+y}{2}\big).(x-y)\big)\big]-v(y)\sin\big (A(\frac{x+y}{2}).(x-y)\big)}{|x-y|^{N+2s}}\mathrm{d}y.
	\end{equation*}\\

\noindent \textit{Claim.} $C_{A}(u,v),  D_{A}(u,w) \in L^{\infty}(\mathbb{R}^N,\mathbb{R})$.\\

	Using the fact that $|u|\in L^{\infty}(\mathbb{R}^N,\mathbb{R})$, $A \in L^{\infty}(\mathbb{R}^N,\mathbb{R}^N)$, $|\sin(t)|, |\cos(t)|<1$, $|\sin(t)|\leq |t|$, $|1-\cos(t)|\leq \frac{t^2}{2}$ for all $t\in \mathbb{R}$ for $x \in \mathbb{R}^N$, $s \in(0, \frac{1}{2})$ and using coarea formula, we deduce
	\begin{eqnarray}
		|C_{A}(u,v)(x)| &\leq& \int_{|x-y|>1}\frac{2\|v\|_{\infty}+\|w\|_{\infty}}{|x-y|^{N+2s}}\mathrm{d}y   \nonumber \\
		& & + \int_{|x-y|<1}\frac{\|A\|_{\infty}^2\|v\|_{\infty}}{2|x-y|^{N+2s-2}}\mathrm{d}y + \int_{|x-y|<1}\frac{\|A\|_{\infty}\|w\|_{\infty}}{|x-y|^{N+2s-1}}\mathrm{d}y \nonumber \\
		&\leq& \omega_{N-1}2\|v\|_{\infty} \int_{1}^{+\infty}\frac{r^{N-1}}{r^{N+2s}}\,\mathrm{d}r+\omega_{N-1}\|w\|_{\infty} \int_{1}^{+\infty}\frac{r^{N-1}}{r^{N+2s}}\,\mathrm{d}r  \nonumber \\
			& & +\frac{\omega_{N-1}}{2}\|A\|_{\infty}^2\|v\|_{\infty} \int_{0}^{1}\frac{r^{N-1}}{r^{N+2s-2}}\,\mathrm{d}r +  \omega_{N-1}\|A\|_{\infty}\|w\|_{\infty}\int_{0}^{1}\frac{r^{N-1}}{r^{N+2s-1}}\,\mathrm{d}r \nonumber \\
		&<& M \nonumber,
	\end{eqnarray}
	for some $M>0$ indepedent of $x$, that is, $C_{A}(u,w) \in L^{\infty}(\mathbb{R}^N,\mathbb{R})$. In similar we verifed that $D_{A}(u,w) \in L^{\infty}(\mathbb{R}^N,\mathbb{R})$.
	Assume that claim is true, then by \eqref{Delta_0} and \eqref{Delta_00} it follow that 
	\begin{equation}\label{C}
		(-\Delta)^{s}v \in L^{\infty}(\mathbb{R}^N,\mathbb{R})
	\end{equation}
	and
	\begin{equation}\label{D}
		(-\Delta)^{s}w \in L^{\infty}(\mathbb{R}^N,\mathbb{R}).
	\end{equation}
	Therefore, invoking Proposition 2.1.9 in \cite{silvetre} to obtain that $v,w \in C^{0,\alpha}(\mathbb{R}^N,\mathbb{R})$ for any $\alpha<2s$, that is, $u\in C^{0,\alpha}(\mathbb{R}^N,\C)$.\\
	
\end{remark}
\begin{remark}\label{decay}
	We can also prove that if $u_{\lambda,k}$ is a solution of \eqref{problem}, then $|u_{\lambda,k}|$ decay
	at infinity. Indeed, in view of Remark \ref{regular} and since  $u_{\lambda,k} \in \mathcal{D}^{s,2}_{A}(\mathbb{R}^N,\C)$, it follows from \cite{stein} that  $|u_{\lambda,k}| \rightarrow 0$ as $|x| \rightarrow +\infty$. 
\end{remark}

 \end{document}